\documentclass[english]{article}
\usepackage[T1]{fontenc}
\usepackage[utf8]{inputenc}
\usepackage{geometry}
\usepackage{mathtools}
\usepackage{amsmath}
\usepackage{amsthm}
\usepackage{amssymb}
\usepackage{hyperref}
\usepackage{xcolor}

\newcommand{\m}[1]{\mathbb{#1}}
\newcommand{\q}[1]{\mathcal{#1}}
\renewcommand{\le}{\leqslant}
\renewcommand{\ge}{\geqslant}

\makeatletter
\theoremstyle{plain}
\newtheorem{thm}{\protect\theoremname}
\theoremstyle{definition}
\newtheorem{defn}[thm]{\protect\definitionname}
\theoremstyle{remark}
\newtheorem{rem}[thm]{\protect\remarkname}
\theoremstyle{plain}
\newtheorem{prop}[thm]{\protect\propositionname}
\theoremstyle{plain}
\newtheorem{lem}[thm]{\protect\lemmaname}

\makeatother

\usepackage{babel}

\providecommand{\definitionname}{Definition}
\providecommand{\lemmaname}{Lemma}
\providecommand{\propositionname}{Proposition}
\providecommand{\remarkname}{Remark}
\providecommand{\theoremname}{Theorem}

\begin{document}

\title{Improved uniqueness of multi-breathers of the modified Korteweg-de Vries
equation}
\author{Raphaël Côte and Alexander Semenov}
\maketitle

\begin{flushright}
\emph{Dedicated to Professor Carlos Kenig on the occasion of his 70th birthday}
\end{flushright}

\bigskip

\begin{abstract}
We consider multi-breathers of \eqref{mKdV}. In \cite{key-49}, a smooth multi-breather was constructed, and proved to be unique in two cases: first, if the class of super-polynomial convergence to the profile (in the spirit of \cite{key-61}), and second, under the assumption that all speeds of the breathers involved are positive (without rate of convergence).

The goal of this short note is to improve the second result: we show that uniqueness still holds
if at most one velocity is negative or zero.
\end{abstract}

\section{Setting of the problem}

\subsection{The modified Korteweg-de Vries equation}


We consider the modified Korteweg-de Vries equation on $\mathbb{R}$:
\begin{equation}
\tag{mKdV}\begin{cases}
\begin{array}{lc}
u_{t}+(u_{xx}+u^{3})_{x}=0 & \quad(t,x)\in\mathbb{R}^{2}\\
u(0)=u_{0} & \quad u(t,x)\in\mathbb{R}
\end{array}\end{cases}\label{mKdV}
\end{equation}

The \eqref{mKdV} equation appears as a model
of some physical problems as  plasma physics, electrodynamics \cite{key-43}, fluid mechanics, ferromagnetic vortices, and more; we refer to \cite{key-49} for further information about the physical applications. Let us recall that
\eqref{mKdV} is globally well-posed for any initial data in $H^{2}$ (see \cite{key-19} for much stronger results), and for such data, three quantities are conserved in time:
\begin{align}
\text{the $L^2$ mass} \qquad 
M[u](t)& :=\frac{1}{2}\int u^{2}(t,x) dx, \\
\text{the energy} \qquad
E[u](t) & := \int \left( \frac{1}{2} u_{x}^{2}(t,x) - \frac{1}{4} u^{4}(t,x)\right) dx, \\
\text{the second energy} \qquad 
F[u](t)& :=\int \left( \frac{1}{2} u_{xx}^{2}(t,x) -\frac{5}{2} u^{2}(t,x)u_{x}^{2}(t,x) +\frac{1}{4} u^{6}(t,x) \right)  dx.
\end{align}
Finally, \eqref{mKdV} is an integrable system, and there are (at least for smooth solutions) infinitely many conservation laws, but we point out that we will only use $H^2$ regularity and the above conservation laws.

\subsection{Solitons and breathers of \texorpdfstring{\eqref{mKdV}}{mKdV}}

The special ``basic'' solutions of (\ref{mKdV}) that we consider
here are solitons and breathers.

\begin{defn}
Let $c>0$, $\kappa\in\{-1,1\}$ and $x_{0}\in\mathbb{R}$. A soliton
$R_{c,\kappa}(x_{0})$ of shape parameter (or velocity) $c$, of sign
$\kappa$ and of translation parameter (or initial position) $x_{0}$
is a solution of (\ref{mKdV}), given by the following formula:
\begin{equation}
\forall (t,x) \in \m R^2, \quad R_{c,\kappa}(t,x;x_{0}):=\kappa Q_{c}(x-x_{0}-ct),
\end{equation}
where $Q_{c}$ is defined by the following formula:
\begin{equation}
\forall x \in \m R, \quad Q_{c}(x):=\left(\frac{2c}{\cosh^{2}(c^{1/2}x)}\right)^{\frac{1}{2}}.
\end{equation}
\end{defn}
\begin{rem}
There exists a constant
$C>0$ that depends only on $c$ such that 
\begin{equation}
\forall (t,x)\in\mathbb{R}^{2}, \quad \lvert R_{c,\kappa}(t,x;x_{0})\rvert\le C\exp(-\sqrt{c}\lvert x-x_{0}-ct\rvert).
\end{equation}
Further properties of solitons (in particular, their $H^2$ variational structure) can be found in \cite{key-49}.
\end{rem}
\begin{defn}
Let $\alpha,\beta>0$ and $x_{1},x_{2}\in\mathbb{R}$. A breather
$B_{\alpha,\beta}(x_{1},x_{2})$ of shape parameters $\alpha,\beta$
and of translation parameters $x_{1},x_{2}$  a solution of (\ref{mKdV}), given by the following
formula: 
\begin{equation}
\forall (t,x)\in\mathbb{R}^{2}, \quad B_{\alpha,\beta}(t,x;x_{1},x_{2}):=2\sqrt{2}\partial_{x}\left[\arctan\left(\frac{\beta}{\alpha}\frac{\sin(\alpha y_{1})}{\cosh(\beta y_{2})}\right)\right],
\end{equation}
where
\begin{align}
y_{1} :=x+\delta t+x_{1}    \quad & \text{and}  \quad y_{2} :=x+\gamma t+x_{2},\\
\text{with}\quad\delta :=\alpha^{2}-3\beta^{2}  \qquad & \text{and} \quad \ \gamma :=3\alpha^{2}-\beta^{2}.
\end{align}
The velocity of $B_{\alpha,\beta}(x_{1},x_{2})$
is $-\gamma=\beta^{2}-3\alpha^{2}$ and its initial position is $-x_{2}$.
\end{defn}
\begin{rem}
There exists a constant $C>0$ that depends only on $\alpha$ and
$\beta$ such that
\begin{equation}
\forall (t,x)\in\mathbb{R}^{2}, \quad \lvert B_{\alpha,\beta}(t,x;x_{1},x_{2})\rvert\le C\exp(-\beta\lvert x+x_{2}+\gamma t\rvert).
\end{equation}
Further properties of breathers can be found in \cite{key-1} or  \cite{key-49}.
\end{rem}

\subsection{Main result}

We consider multi-breathers made of $K$ breathers and $L$ solitons in the sense of \cite{key-49}, which we recall now. Let $K,L\in\mathbb{N}$ and set $J=K+L$. For $1\le k\le K$, let $\alpha_{k},\beta_{k}>0$ and $x_{1,k}^{0},x_{2,k}^{0}\in\mathbb{R}$, and define the breathers
\begin{equation}
B_{k}:=B_{\alpha_{k},\beta_{k}}(x_{1,k}^{0},x_{2,k}^{0}), \quad \text{with velocity} \quad v_{k}^{b}:=\beta_{k}^{2}-3\alpha_{k}^{2}. \label{br}
\end{equation}
For $1\le l\le L$, let $c_{l}>0$, $\kappa_{l}\in\{-1,1\}$ and
$x_{0,l}^{0}\in\mathbb{R}$ and consider the solitons \begin{equation}
R_{l}:=R_{c_{l},\kappa_{l}}(x_{0,l}^{0}) , \quad \text{with velocity} \quad v_{l}^{s}:=c_{l}.\label{sol}
\end{equation}
An essential assumption in the present analysis is that \emph{all the velocities of the considered objects, solitons or breathers,
must be distinct}:
\begin{equation}
\forall k\neq k'\quad v_{k}^{b}\neq v_{k'}^{b},\qquad\forall l\neq l'\quad v_{l}^{s}\neq v_{l'}^{s},\qquad\forall k,l\quad v_{k}^{b}\neq v_{l}^{s}.\label{v_dist}
\end{equation}
We may therefore order the speed and define an increasing function:
\begin{equation}
\underline{v}:\{1,\dots,J\}\to\{v_{k}^{b},1\le k\le K\}\cup\{v_{l}^{s},1\le l\le L\}.\label{v_j}
\end{equation}
The $J$-uple $(v_{1},\dots,v_{J})$ is thus the ordered set of all possible
velocities of our objects. This allows us to write in a convenient way the sum of breathers and solitons: for $1\le j\le J$, we define $P_{j}$
as the object (either the soliton $R_{l}$ or the breather $B_{k}$)
that corresponds to the velocity $v_{j}$, so that $P_{1},\dots,P_{J}$
are ordered by increasing velocity. We consider the sum of these breathers and solitons
\begin{equation}
P:=\sum_{j=1}^{J}P_{j}=\sum_{k=1}^{K}B_{k}+\sum_{l=1}^{L}R_{l}, \label{sum}
\end{equation}
and the associated multi-breather, that are solutions to \eqref{mKdV} which behave as $P$ for large time as defined below

\begin{defn}
\label{def:mb}
Given solitons and breathers (\ref{sol}), (\ref{br}),
whose sum is given by (\ref{sum}), a \emph{multi-breather} associated
to the sum $P$ of solitons and breathers is a solution $p\in\mathcal{C}([T^{*},+\infty),H^{2}(\mathbb{R}))$,
for a constant $T^{*}>0$, of (\ref{mKdV}) such that 
\begin{equation} \label{def:mbH2}
\lVert p(t)-P(t)\rVert_{H^{2}} \to 0 \quad \text{as} \quad t \to +\infty.
\end{equation}
\end{defn}

Let us recall from \cite[Theorem 1.2]{key-49} that as soon as the condition on the speeds \eqref{v_dist} on $P$ is satisfied, there exist a multi-breather $p$ related to $P$, such that $p$ is smooth and the convergence $p-P \to 0$ occurs at an exponential rate in any $H^s$, $s \in \m R$. 

Also, from \cite[Theorem 1.4]{key-49}, under the extra assumption that all speed are positive, that is $v_1 >0$, this multi-breather is unique in the class of \eqref{def:mbH2}.
\bigskip

In this note, we revisit the proof given in \cite{key-49}, in order to improve this last uniqueness result: it actually holds if at most one of the velocities is non positive, that is under the assumption $v_2 >0$. Here is the precise statement.

\begin{thm}
\label{thm:uniq}
Given breathers \eqref{br} and solitons \eqref{sol},
whose velocities satisfy \eqref{v_dist}, let $P$ be the sum of the considered solitons and breathers given
in \eqref{sum}.

Assume that $v_{2}>0$, so that all velocities, except
possibly one, are positive. Then the multi-breather of \cite{key-49}
$p \in \q C([T^*,+\infty),H^2)$ associated to $P$ is the \emph{unique} solution of \eqref{mKdV} such that \eqref{def:mbH2} holds.
\end{thm}

Multi-solitons have been constructed for many dispersive models (see for example \cite{key-3,key-16,key-11,key-66} for (NLS), Klein-Gordon, or water-waves), but the question of uniqueness (or classification) is generally open. The examples that we are aware of, were such uniqueness is known, are the generalized Korteweg-de Vries equation \cite{key-2}, the generalized Benjamin-Bona-Mahony equation \cite{EDM04} and the Zakharov-Kuznetsov equation \cite{key-60}. The underlying difficulty is the interaction of the nonlinear object and linear dispersion: for the two models, solitons move to the right, and dispersion to the left, which allows a very nice decoupling, that one can express via a monotonicity property. This feature is however absent in other dispersive models, which explains why uniqueness of multi-solitons remains an open problem in general.

In this note, we consider the more complex multi-breather of \eqref{mKdV}. The point is that breathers may travel to the left, in the Airy dispersion zone: but if at most one moves there, our result shows that uniqueness still holds.

\section{Proof}

In order to prove Theorem \eqref{thm:uniq}, we consider $P$ as in \eqref{sum}, we assume that $v_2 >0$, and we consider a multi-breather $u \in \q C([T,+\infty),H^2)$ such that
\[ \| u(t) - P(t) \|_{H^2} \to 0 \quad \text{as} \quad t \to +\infty. \]
Our goal is to prove that $u\equiv p$, the multi-breather associated to $P$ constructed in \cite[Theorem 1.2]{key-49}. For this, the main step is to prove that $u(t)$ converges actually exponentially fast to its profile $P(t)$.

\begin{prop}
\label{prop:uniq}
There exists $\varpi>0$, $T_{0}\ge T$ and $C>0$ such that
\begin{equation}
\forall t \ge T_0, \quad \lVert u(t)-P(t)\rVert_{H^{2}}\le Ce^{-\varpi t}.
\end{equation}
\end{prop}

This corresponds to Proposition 4.10 in \cite{key-49} (where of course the assumption on $P$ is different: there one suppose that $v_1 >0$.).

\subsection{Proof of Proposition \ref{prop:uniq}}

The proof follows mainly the lines of Section 4.2 of \cite{key-49}, with several changes that we will detail here. An important ingredient is (almost) monotonicity properties of localized quantities.

We do our best to treat breather and solitons together. To this end, we define the shape parameters as follows: for $j=1,\dots,J$,
\begin{itemize}
\item if $P_j=B_k$ is a breather, then
\begin{equation} \label{def:shapeB}
(a_j,b_j)=(\alpha_k,\beta_k),
\end{equation}
\item if $P_j=R_l$ is a soliton, then
\begin{equation}
(a_j,b_j)=(0,\sqrt{c_l}).  \label{def:shapeS}
\end{equation}
\end{itemize}

The shape of the cut-off function that we will use is given by $\Psi$:
\begin{equation}
\Psi(x):=\frac{2}{\pi}\arctan\big(\exp(-\sqrt{\sigma}x/2)\big),
\end{equation}
where $\sigma>0$ is small enough (and precise conditions will be given in the proof).

We consider a cut-off function $\Phi_j$ given for $j=1,\dots,J-1$ by
\begin{equation}
\Phi_{j}(t,x):=\Psi(x-m_{j}t),
\end{equation}
and $\Phi_{J} \equiv 1$.
For $j \in \{ 1, \dots, J-1 \}$, $\Psi_j$ tend to $1$ at $-\infty$, to $0$ at $+\infty$, with an exponentially localized transition between the centers of $P_{j}$ and $P_{j+1}$. This requires that for all such $j$, $v_{j} < m_j < v_{j+1}$.

However, in order to the monotonicity argument
to work, we need to choose cut-off functions that 
move to the right, i.e. have positive velocities $m_j >0$. For $j =2, \dots, J-1$, we set
\begin{equation}
m_j:=\frac{v_{j}+v_{j+1}}{2}.
\end{equation} 
(Then we indeed have $m_j>0$ because $v_{j} \ge v_2 >0$).
On the other hand, $m_{1}$ needs better tuning: we define $m_{1}$ in the following way.
\begin{enumerate}
\item We first set $0<\nu_1<1$ such that
\begin{equation}
\label{nu1}
(b_{1}^2-a_{1}^2)+\nu_1(a_{1}^2+b_1^2)>0,
\end{equation}
\item then we choose $m_1$ such that
\begin{align}
\max(0,v_1) & <m_{1}<v_2, \quad  \text{and} \\
\label{m2def}
m_1(b_1^2-a_1^2) & >\frac{1}{2}(\nu_1-1)(a_1^2+b_1^2)^2.
\end{align}
\end{enumerate}
Condition \eqref{nu1} corresponds to a choice of $\nu_1$ sufficiently near $1$.
Condition \eqref{m2def} can be satisfied: indeed, if $b_1^2-a_1^2 \le 0$ then $v_1<0$ and it suffices to choose $m_1>0$ sufficiently small; if $b_1^2-a_1^2 >0$, it is always satisfied as the righthand side is negative.

These conditions will be used in order to derive suitable monotonicity properties, see Step 3.

\bigskip

We set $\tau_{0}>0$ the minimal distance between $\{v_{1},\dots,v_{J}\}$
and $\{m_{1},\dots,m_{J-1}\}$.

As mentioned, the scheme of the proof is here roughly the same as in \cite{key-49}; however, some specifics change.

First we modulate the breathers $P_j$ into $\widetilde P_j$ (by translation only), so that the default $w = u - \widetilde P$ (defined in Lemma \ref{lem:mod}) enjoys orthogonality properties. Then we prove that $\lVert w(t)\rVert_{H^{2}}\le Ce^{-\varpi t}$
by induction, where $\varpi>0$ is a constant depending on the data of the problem.

For $j=1,\dots,J$, proposition $\mathcal{P}_{j}$ reads
\begin{equation} \label{def:rec}
\forall t \ge T', \quad \int\big(w^{2}+w_{x}^{2}+w_{xx}^{2}\big)\Phi_{j} +  \sum_{i=1}^{j} \left| \int \widetilde P_{i} w \right| \le C_j e^{-2\varpi t},
\end{equation}
where $T' \ge T$ is to be defined in the proof. $\mathcal P_0$ is the assertion ``True''. 

Given $j \in \{ 1, \dots, J \}$, we assume $\mathcal{P}_{j-1}$, and our goal is to prove $\mathcal{P}_{j}$.
We finally infer $\lVert u(t) - P(t) \rVert_{H^{2}}\le Ce^{-\varpi t}$, in the concluding step of the proof.

\bigskip

One difference with \cite{key-49} is that here we make our proof by induction on the
modulated difference $w$ and not on $u-P$. This is not crucial, but we find it nicer to perform a modulation for all the objects at once. One key difference compared to \cite{key-49}, though, 
is that we need and prove monotonicity for a functional which is slightly weaker than the natural Lyapunov functional required for the proof. When $v_1 <0$, the proof requires a careful interpolation between positive terms in order to balance negative terms.

\subsubsection*{Step 1: Modulation}

This step is devoted to the proof of the following modulation lemma.
\begin{lem} \label{lem:mod}
There exists $C>0$, $T_{2}\ge T$, such that there exist unique
$\mathcal{C}^{1}$ functions $y_{1,k},y_{2,k},y_{0,l}:[T_{2},+\infty)\to\mathbb{R}$
such that if we set:
\begin{equation}
w(t,x):=u-\widetilde P ,
\end{equation}
where
\begin{equation}
\widetilde P (t,x):=\sum_{k=1}^K\widetilde B_k(t,x)+\sum_{l=1}^L\widetilde R_l (t,x),
\end{equation}
and
\begin{equation}
\widetilde R_l (t,x):=\kappa_{l}Q_{c_{l}}(x-x_{0,l}^{0}+y_{0,l}(t)-c_{l}t),
\end{equation}
\begin{equation}
\widetilde B_k (t,x):=B_{\alpha_{k},\beta_{k}}(t,x;x_{1,k}+y_{1,k}(t),x_{2,k}+y_{2,k}(t)),
\end{equation}

then, $w(t)$ satisfies, for any $t\in[T_{2},+\infty)$,
\begin{equation}
\forall l=1,\dots,L,\quad\int (\widetilde R_l)_{x}(t)w(t)=0,
\end{equation}
\begin{equation}
\forall k=1,\dots,K\quad\int (\widetilde B_k)_{1}(t)w(t)=\int (\widetilde B_k)_{2}(t)w(t)=0,
\end{equation}
where we denote:
\begin{equation}
(\widetilde B_k)_{1} (t,x):=\partial_{x_{1}}\widetilde B_k,\quad (\widetilde B_k)_{2}(t,x):=\partial_{x_{2}}\widetilde B_k .
\end{equation}

Moreover, for any $t\in[T_{2},+\infty)$,
\begin{equation}
\label{mod_bou}
\lVert w(t)\rVert_{H^{2}}+\lvert y_{1,k}(t)\rvert+\lvert y_{2,k}(t)\rvert+\lvert y_{0,l}(t)\rvert\le C\lVert v(t)\rVert_{H^{2}},
\end{equation}
and, if $\varpi$ is small enough,
\begin{equation}
\label{mod_der}
\lvert y_{1,k}'(t)\rvert+\lvert y_{2,k}'(t)\rvert+\lvert y_{0,l}'(t)\rvert\le C\bigg(\int w(t)^{2}\Phi_{j}\bigg)^{1/2}+Ce^{-\varpi t}.
\end{equation}
\end{lem}
\begin{proof}
The proof of this lemma can be performed in the same manner as in
\cite[Lemma 2.8]{key-49}.
\end{proof}

As above for \eqref{sum}, we denote $\tilde P_j = \tilde B_k$ if $P_j = B_k$ is a breather, and $\tilde P_j = \tilde R_l$ if $P_j=R_l$ is a soliton, so that
\[ \tilde P = \sum_{j=1}^J \tilde P_j. \]

The difference with \cite{key-49} is that
\begin{itemize}
\item the modulation that we perform here does not modify any shape parameter (that is why there is only one modulation direction for each soliton here),
\item we perform the modulation once and not on each step of the induction.
\end{itemize}

\subsubsection*{Step 2: Approximation of the Lyapunov functional}

This step is devoted to define a localized Lyapunov
functional. Let $j \in \{ 1, \dots, J \}$. First, we define 
 the localized conservation laws are defined as follows:
\begin{align}
M_j(t) & :=\int u^2(t)\Phi_j(t), \\
E_j(t) & :=\int \bigg[\frac{1}{2}u_x^2-\frac{1}{4}u^4\bigg]\Phi_j(t), \\
F_j(t) & :=\int\bigg[\frac{1}{2}u_{xx}^2-\frac{5}{2}u^2u_x^2+\frac{1}{4}u^6\bigg]\Phi_j(t).
\end{align}

Then  the localized Lyapunov functional is
\begin{equation}
\mathcal{H}_{j}(t):=F_{j}(t)+2\big(b_{j}^{2}-a_{j}^{2}\big)E_{j}(t)+\big(a_{j}^{2}+b_{j}^{2}\big)^{2}M_{j}(t),
\end{equation}
where $a_j,b_j$ stand for generalized shape parameters defined in \eqref{def:shapeB}-\eqref{def:shapeS}.

This Lyapunov functional was already introduced in \cite{key-49}.
We prove the following lemma that somehow quantifies how far is a modulated sum of solitons and breathers from being a critical point for $\mathcal{H}_{j}$.

First, there hold the following Taylor expansions.

\begin{lem} \label{lem:taylor_1}
There exists $C>0$, $T_{1} \ge T$
such that the following holds for any $t\ge T_{1}$:

\begin{gather} 
\bigg\lvert M_{j}(t)-\sum_{i=1}^{j}M\big[\widetilde P_{i}\big]-\sum_{i=1}^{j}\int\widetilde P_{i}w-\frac{1}{2}\int w^{2}\Phi_{j}\bigg\rvert  \le  Ce^{-2\varpi t},\label{1eq:-155} \\
 \bigg\lvert E_{j}(t)-\sum_{i=1}^{j}E\big[\widetilde P_{i}\big]-\sum_{i=1}^{j}\int\Big[ (\widetilde P_{i})_{x}w_{x}-\widetilde P_{i}^{3}w\Big] -\int\Big[\frac{1}{2}w_{x}^{2}-\frac{3}{2}\widetilde P^{2} w^{2}\Big]\Phi_{j}\bigg\rvert  \le  Ce^{-2\varpi t}+ o\left( \int w^{2}\Phi_{j} \right) ,\label{1eq:-156} \\
\bigg\lvert F_{j}(t)-\sum_{i=1}^{j}F\big[\widetilde P_{i}\big]-\sum_{i=1}^{j}\int\Big[(\widetilde P_{i})_{xx}w_{xx}-5\widetilde P_{i}(\widetilde P_{i})_{x}^{2}w-5 {\widetilde P_{i}}^{2} (\widetilde P_{i})_{x}w_{x}+\frac{3}{2} \widetilde P_{i}^{5}w\Big] \nonumber \\
  \quad -\int\Big[\frac{1}{2}w_{xx}^{2} -\frac{5}{2}w^{2} \widetilde P_{x}^{2}-10\widetilde P w \widetilde P_{x}w_{x}-\frac{5}{2}\widetilde P^{2} w_{x}^{2}+\frac{15}{4} \widetilde P ^{4} w^{2}\Big]\Phi_{j}(t)\bigg\rvert\\ \le  Ce^{-2\varpi t}+o\left( \int\big(w^{2}+w_{x}^{2}\big)\Phi_{j} \right).\label{1eq:-157} 
\end{gather}
\end{lem}

\begin{proof}
See for example \cite[Proposition 2.12]{key-49}. We emphasize that we do not use here the induction assumption.
\end{proof}

The Lyapunov functional is constructed so as to make the linear terms in $w$ cancel, as seen below.

\begin{lem} \label{1lem:dl}
There exists $T_{2} \ge  T_1$ such that the following holds for $t\ge  T_{2}$: 
\begin{align} 
\mathcal{H}_{j}(t) & =\sum_{i=1}^{j}F[\widetilde P_{i} ]+2\big(b_{j}^{2}-a_{j}^{2}\big)\sum_{i=1}^{j}E[\widetilde P_{i} ]+\big(a_{j}^{2}+b_{j}^{2}\big)^{2}\sum_{i=1}^{j}M[\widetilde P_{i} ] \nonumber \\  
& \qquad +H_{j}(t)+O(e^{-2\varpi t})+ o \bigg(\int\big(w^{2}+w_{x}^{2} \big)\Phi_{j}\bigg),\label{1eq:-160}
\end{align}
where  
\begin{align} 
H_{j}(t): & =\int\Big[\frac{1}{2}w_{xx}^{2}-\frac{5}{2}w_{x}^{2}\widetilde P_{j}^{2}+\frac{5}{2}w^{2} (\widetilde P_{j})_{x}^{2}+5w^{2}\widetilde P_{j} (\widetilde P_{j})_{xx} +  \frac{15}{4}w^{2} \widetilde P_{j}^{4}\Big]\Phi_{j}(t) \nonumber \\
& \quad + \big(b_{j}^{2}-a_{j}^{2}\big)\int\Big[w_{x}^{2}-3w^{2}\widetilde P_{j}^{2}\Big]\Phi_{j}(t) +\frac{1}{2}\big(a_{j}^{2}+b_{j}^{2}\big)^{2}\int w^{2}\Phi_{j}(t).
\label{1eq:-161}
 \end{align} 
 \end{lem}

\begin{proof}
The proof of the lemma above can be performed as in
\cite[Proposition 2.12]{key-49}: it uses Lemma \ref{lem:taylor_1}, the elliptic equation satisfied by $P_{j}$ or $\widetilde P_{j}$ (we
do not need to make a distinction whether it is a soliton
or a breather here), and the induction assumption for the contributions in the region $x \le m_{j-1}t$.
\end{proof}

\subsubsection*{Step 3: Monotonicity}

Let us first recall some monotonicity properties related to the localized conservations laws.
 \begin{lem}
 \label{mon}
Let $\omega >0$ as small as desired. There exists $T_{3} = T_3(\omega) \ge T_2$
and $C>0$ such that for $t\ge T_{3}$,
\begin{gather}
\sum_{i=1}^{j}M[P_{i}]-M_{j}(t)\ge -Ce^{-2\varpi t}, \label{est:mono_M} \\
\sum_{i=1}^{j}\big(E[P_{i}]+\omega M[P_{i}]\big)-\big(E_{j}(t)+\omega M_{j}(t)\big)\ge -Ce^{-2\varpi t}, \label{est:mono_E} \\
\sum_{i=1}^{j}\big(F[P_{i}]+\omega M[P_{i}]\big)-\big(F_{j}(t)+\omega M_{j}(t)\big)\ge -Ce^{-2\varpi t}. \label{est:mono_F}
\end{gather}
\end{lem}

\begin{proof}
The lemma above may be proved in the same manner as in \cite[Lemma 4.11]{key-49}.
\end{proof}

We emphasize that some extra $L^2$ mass is needed in order to gain monotonicity for $E_j$ or $F_j$: this fact was already noted in \cite{key-2}, and is related to a lack of control of non linear terms far away from the breathers/solitons.

\bigskip

We now turn to the main monotonicity result that we will use.

Let $0<\nu<1$ be close enough to $1$, to be fixed later. We define, for $j=1,\dots,J$, a functional $\mathcal{F}_j$ that is close to the Lyapunov functional $\mathcal H_j$: 
\begin{equation} 
\mathcal{F}_j(t):=F_j(t)+2(b_{j}^2-a_{j}^2)E_j(t)+\nu(a_{j}^2+b_{j}^2)^2M_j(t). 
\end{equation}

The following lemma states the almost-growth of $\mathcal{F}_j$:
\begin{lem}
\label{1F2} There exists $0<\nu<1$ close enough to $1$ such that there exists $T_4 \ge T_2$ and $C>0$ such that for any $t\ge T_4$, \begin{equation} \label{est:1F2}
 \mathcal{F}_j(t)-\sum_{i=1}^{j}F[P_{i}]-2(b_{j}^2-a_{j}^2)\sum_{i=1}^{j}E[P_{i}]-\nu(a_{j}^2+b_{j}^2)^2\sum_{i=1}^{j}M[P_{i}]\le Ce^{-2\varpi t}. \end{equation}
\end{lem}

\begin{rem}
We emphasize the factor $\nu<1$ in front of the $M_j$ term, which make it a \emph{weakened} version of the Lyapunov functional $\mathcal H_j$. Therefore, as $M_j$ enjoys strong monotonicity,  the monotonicity of $\mathcal{F}_j$ is a stronger result than merely that
 of $\mathcal{H}_j$. This improvement is needed in order to deal with $\displaystyle \int \widetilde P_{j} w$, see Step 5.
\end{rem}

 \begin{proof}
If $b_{j}^2-a_{j}^2\ge 0$, then Lemma \ref{1F2} is an immediate consequence of Lemma \ref{mon}.
For the rest of the proof, we consider the case $b_{j}^2-a_{j}^2<0$, which can only occur when $j=1$, which we assume for the rest of this proof. Let
\begin{equation}
\nu=\nu_1+\frac{2}{3}(1-\nu_1) <1,
\end{equation}
where $\nu_1$ is defined in \eqref{nu1}.

In the proof of Lemma \ref{mon} (see \cite[Lemma 4.11]{key-49}), there hold the more precise bounds: given $\omega>0$, there exist $T_3' =T_3'(\omega) \ge T_2$ such that for all $t \ge T_3'$
\begin{align}
\frac{d}{dt}F_{j}(t) & \ge -Ce^{-2\varpi t}+\frac{3}{2}\int u_{xxx}^{2}|\Phi_{jx}|+\frac{m_{j}}{2}\int u_{xx}^{2}|\Phi_{jx}| - \omega\int (u_{xx}^{2} + u_{x}^{2} + u^{2}) |\Phi_{jx}|, \\
 -\frac{d}{dt}E_{j}(t) & \ge -Ce^{-2\varpi t}-\frac{3}{2}\int u_{xx}^{2}|\Phi_{jx}| -\frac{m_{j}}{2}\int u_{x}^{2}|\Phi_{jx}| -\omega \int (u_{x}^{2} + u^{2}) |\Phi_{jx}|, \\ 
 \frac{d}{dt}M_{j}(t) & \ge -Ce^{-2\varpi t}+\frac{3}{2}\int u_{x}^{2}|\Phi_{jx}|+\frac{m_{j}}{2}\int u^{2}|\Phi_{jx}| -\omega\int u^{2}|\Phi_{jx}|.%
 \end{align}%
Summing up the right linear combination, we infer:
\begin{align} 
\frac{d}{dt}\mathcal{F}_{j}(t) & \ge -Ce^{-2\varpi t}+\frac{3}{2}\int u_{xxx}^{2}\lvert\Phi_{jx}\rvert \nonumber \\ 
& \quad + \Big(3(b_{j}^{2}-a_{j}^{2})+\frac{m_{j}}{2}-\omega\Big)\int u_{xx}^{2}\lvert\Phi_{jx}\rvert \nonumber \\ 
& \quad + \Big(\frac{3}{2}\nu(a_{j}^{2}+b_{j}^{2})^{2}+m_{j}(b_{j}^{2}-a_{j}^{2})-\omega\Big)\int u_{x}^{2}\lvert\Phi_{jx}\rvert \nonumber \\  
& \quad + \Big(\frac{m_{j}}{2}\nu(a_{j}^{2}+b_{j}^{2})^{2}-\omega\Big)\int u^{2}\lvert\Phi_{jx}\rvert. 
\end{align}
Heuristically, $\omega$ can be neglected, so that the coefficients of the terms in $\int u_{xxx}^{2}\lvert\Phi_{jx}\rvert$, $\int u_{x}^{2}\lvert\Phi_{jx}\rvert$
and $\int u^{2}\lvert\Phi_{jx}\rvert$ are all positive: only the term in $\int u_{xx}^{2}\lvert\Phi_{jx}\rvert$ might be problematic, and will concentrate our efforts.

From the definition of $m_2$ given by \eqref{m2def}, we have that
\begin{equation} \frac{3}{2}\nu(a_{j}^{2}+b_{j}^{2})^{2}+m_{j}(b_{j}^{2}-a_{j}^{2})>\frac{3}{2}\nu'(a_{j}^{2}+b_{j}^{2})^{2}, \end{equation} 
where
\begin{equation}
\nu'=\nu_1+\frac{1-\nu_1}{3}.
\end{equation}

We choose $\omega$ small enough with respect to the previous choice (by choosing $T_4$ large enough) so that 
\begin{equation} 
\frac{3}{2}\nu(a_{j}^{2}+b_{j}^{2})^{2}+m_{j}(b_{j}^{2}-a_{j}^{2})-\omega\ge \frac{3}{2}\nu'(a_{j}^{2}+b_{j}^{2})^{2}. 
\end{equation}
Knowing that $a_{j}^2+b_{j}^2>0$, we may choose $\omega$ even smaller (with respect to $m_{j}$ and $\nu$) so that \begin{equation} \frac{m_{j}}{2}\nu(a_{j}^{2}+b_{j}^{2})^{2}-\omega\ge 0, \end{equation} and \begin{equation} 
3(b_{j}^{2}-a_{j}^{2})+\frac{m_{j}}{2}-\omega\ge 3(b_{j}^{2}-a_{j}^{2}). \end{equation}
In the case when with the chosen values of $m_{j}$, $\nu$ and $\omega$, $3(b_{j}^{2}-a_{j}^{2})+\frac{m_{j}}{2}-\omega$ is positive, the desired conclusion is straightforward by integration. From now on, we place ourselves in the case when 
\begin{equation} 
3(b_{j}^{2}-a_{j}^{2})+\frac{m_{j}}{2}-\omega<0. 
\end{equation}
Now, we want to bound above $\int u_{xx}^{2}\lvert\Phi_{jx}\rvert$. By integration by parts, \begin{equation}\begin{aligned} \int u_{xx}^{2}\lvert\Phi_{jx}\rvert & =-\int u_{x}u_{xxx}\lvert\Phi_{jx}\rvert-\int u_{x}u_{xx}\lvert\Phi_{jxx}\rvert\\  & \le \sqrt{\int u_{x}^{2}\lvert\Phi_{jx}\rvert\int u_{xxx}^{2}\lvert\Phi_{jx}\rvert}+\frac{\sqrt{\sigma}}{2}\sqrt{\int u_{x}^{2}\lvert\Phi_{jx}\rvert\int u_{xx}^{2}\lvert\Phi_{jx}\rvert}, \end{aligned}\end{equation} because $\lvert\Phi_{jxx}\rvert\le \frac{\sqrt{\sigma}}{2}\lvert\Phi_{jx}\rvert$.
We denote: \begin{equation} X:=\sqrt{\int u_{xx}^{2}\lvert\Phi_{jx}\rvert}, \end{equation} and \begin{equation} A:=\sqrt{\int u_{x}^{2}\lvert\Phi_{jx}\rvert\int u_{xxx}^{2}\lvert\Phi_{jx}\rvert}. \end{equation} So, we have that \begin{equation} X^{2}\le A+\varepsilon X, \end{equation} where \begin{equation} \varepsilon:=\frac{\sqrt{\sigma}}{2}\sqrt{\int u_{x}^{2}\lvert\Phi_{jx}\rvert}\le \frac{\sqrt{\sigma}}{2}\lVert u\rVert_{\dot{H^{1}}}, \end{equation} which can be as small as we want if we take $\sigma$ small enough (for a given solution $u$).
We deduce that 
\begin{equation}
 X\le \frac{\varepsilon+\sqrt{\varepsilon^{2}+4A}}{2}\le \varepsilon+\sqrt{A}. 
 \end{equation} 
Thus, 
\begin{equation}
\begin{aligned} \int u_{xx}^{2}\lvert\Phi_{jx}\rvert&\le \Bigg(\frac{\sigma}{4}\sqrt{\int u_{x}^{2}\lvert\Phi_{jx}\rvert}+\sqrt{\sigma}\bigg(\int u_{x}^{2}\lvert\Phi_{jx}\rvert\int u_{xxx}^{2}\lvert\Phi_{jx}\rvert\bigg)^{\frac{1}{4}}\\ &\qquad+\sqrt{\int u_{xxx}^{2}\lvert\Phi_{jx}\rvert}\Bigg)\sqrt{\int u_{x}^{2}\lvert\Phi_{jx}\rvert}. \end{aligned}
\end{equation} 
So, 
\begin{equation}
\begin{aligned} 
\MoveEqLeft\Big(3(b_{j}^{2}-a_{j}^{2})+\frac{m_{j}}{2}-\omega\Big)\int u_{xx}^{2}\lvert\Phi_{jx}\rvert  \ge 3(b_{j}^{2}-a_{j}^{2})\sqrt{\int u_{x}^{2}\lvert\Phi_{jx}\rvert\int u_{xxx}^{2}\lvert\Phi_{jx}\rvert}\\  
& +3(b_{j}^{2}-a_{j}^{2})\sqrt{\sigma}\bigg(\int u_{x}^{2}\lvert\Phi_{jx}\rvert\bigg)^{\frac{3}{4}}\bigg(\int u_{xxx}^{2}\lvert\Phi_{jx}\rvert\bigg)^{\frac{1}{4}}\\  
& +3(b_{j}^{2}-a_{j}^{2})\frac{\sigma}{4}\int u_{x}^{2}\lvert\Phi_{jx}\rvert. 
\end{aligned}
\end{equation}
On the other hand, we have, for a choice of $\nu_{2},\nu_{3}>0$ such that $\nu_{1}+\nu_{2}+\nu_{3}=\nu'$ 
 that 
 \begin{equation}
 \begin{aligned} 
 \MoveEqLeft\frac{3}{2}\int u_{xxx}^{2}\lvert\Phi_{jx}\rvert+\Big(\frac{3}{2}\nu'(a_{j}^{2}+b_{j}^{2})^{2}\Big)\int u_{x}^{2}\lvert\Phi_{jx}\rvert\\  
 &  \ge \frac{3}{2}\nu_{1}\int u_{xxx}^{2}\lvert\Phi_{jx}\rvert+\Big(\frac{3}{2}\nu_{1}(a_{j}^{2}+b_{j}^{2})^{2}\Big)\int u_{x}^{2}\lvert\Phi_{jx}\rvert\\  
 & +\frac{3}{2}(1-\nu_1) \int u_{xxx}^{2}\lvert\Phi_{jx}\rvert+\Big(\frac{3}{2}\nu_{2}(a_{j}^{2}+b_{j}^{2})^{2}\Big)\int u_{x}^{2}\lvert\Phi_{jx}\rvert\\  
 & +\Big(\frac{3}{2}\nu_{3}(a_{j}^{2}+b_{j}^{2})^{2}\Big)\int u_{x}^{2}\lvert\Phi_{jx}\rvert\\  
 & \ge 2\sqrt{\frac{3}{2}\nu_{1}\Big(\frac{3}{2}\nu_{1}(a_{j}^{2}+b_{j}^{2})^{2}\Big)}\sqrt{\int u_{x}^{2}\lvert\Phi_{jx}\rvert\int u_{xxx}^{2}\lvert\Phi_{jx}\rvert}\\  
 & + 4\bigg(\frac{3}{2}(1-\nu_1) \int u_{xxx}^{2}\lvert\Phi_{jx}\rvert\bigg)^{\frac{1}{4}}\bigg(\Big(\frac{1}{2}\nu_{2}(a_{j}^{2}+b_{j}^{2})^{2}\Big)\int u_{x}^{2}\lvert\Phi_{jx}\rvert\bigg)^{\frac{3}{4}}\\  
 & +\Big(\frac{3}{2}\nu_{3}(a_{j}^{2}+b_{j}^{2})^{2}\Big)\int u_{x}^{2}\lvert\Phi_{jx}\rvert\\  
 & \ge 3 \nu_{1}(a_{j}^{2}+b_{j}^{2})\sqrt{\int u_{x}^{2}\lvert\Phi_{jx}\rvert\int u_{xxx}^{2}\lvert\Phi_{jx}\rvert}\\  
 & +2\cdot3^{\frac{1}{4}}(1-\nu_1)^{\frac{1}{4}}\nu_{2}^{\frac{3}{4}}(a_{j}^{2}+b_{j}^{2})^{\frac{3}{2}}\bigg(\int u_{x}^{2}\lvert\Phi_{jx}\rvert\bigg)^{\frac{3}{4}}\bigg(\int u_{xxx}^{2}\lvert\Phi_{jx}\rvert\bigg)^{\frac{1}{4}}\\  
 & +\frac{3}{2}\nu_{3}(a_{j}^{2}+b_{j}^{2})^{2}\int u_{x}^{2}\lvert\Phi_{jx}\rvert. 
 \end{aligned}
 \end{equation}
This is why, we deduce that 
\begin{equation} \label{est:Fend}
\begin{aligned} 
\MoveEqLeft\frac{d}{dt}\mathcal{F}_{j}(t)  \ge -Ce^{-2\varpi t}\\ &+\Big(3(b_{j}^{2}-a_{j}^{2})+3 \nu_{1} (a_{j}^{2}+b_{j}^{2})\Big)\sqrt{\int u_{x}^{2}\lvert\Phi_{jx}\rvert\int u_{xxx}^{2}\lvert\Phi_{jx}\rvert}\\  
& +\Big(3(b_{j}^{2}-a_{j}^{2})\sqrt{\sigma}+2\cdot3^{\frac{1}{4}}(1-\nu_1)^{\frac{1}{4}}\nu_{2}^{\frac{3}{4}}(a_{j}^{2}+b_{j}^{2})^{\frac{3}{2}}\Big)\bigg(\int u_{x}^{2}\lvert\Phi_{jx}\rvert\bigg)^{\frac{3}{4}}\bigg(\int u_{xxx}^{2}\lvert\Phi_{jx}\rvert\bigg)^{\frac{1}{4}}\\  
& +\Big(3(b_{j}^{2}-a_{j}^{2})\frac{\sigma}{4}+\frac{3}{2}\nu_{3}(a_{j}^{2}+b_{j}^{2})^{2}\Big)\int u_{x}^{2}\lvert\Phi_{jx}\rvert. 
\end{aligned}
\end{equation}
To finish, we remark that the coefficient in front of the integrals in \eqref{est:Fend} are all non negative: indeed 
\begin{equation} 
3(b_{j}^{2}-a_{j}^{2})+3\nu_{1}(a_{j}^{2}+b_{j}^{2})\ge 0, 
\end{equation}
by definition of $\nu_1$ given in \eqref{nu1};
 \begin{equation} 
3(b_{j}^{2}-a_{j}^{2})\sqrt{\sigma}+2\cdot3^{\frac{1}{4}} (1-\nu_1)^{\frac{1}{4}}\nu_{2}^{\frac{3}{4}}(a_{j}^{2}+b_{j}^{2})^{\frac{3}{2}}\ge 0, \end{equation} 
 and 
 \begin{equation} 
 3(b_{j}^{2}-a_{j}^{2})\frac{\sigma}{4}+\frac{3}{2}\nu_{3}(a_{j}^{2}+b_{j}^{2})^{2}\ge 0, \end{equation}
 by choosing $\sigma>0$ small enough.

Thus, \begin{equation} \frac{d}{dt}\mathcal{F}_{j}(t)\ge -Ce^{-2\varpi t}. \end{equation} We obtain the desired conclusion by integration. \end{proof}

\subsubsection*{Step 4: Bound from above for $H_{j}(t)$}

\begin{lem}
For any $t\ge T_4$, we have that
\begin{equation}
H_{j}(t)\le Ce^{-2\varpi t}+o \left( \int\big(w^{2}+w_{x}^{2}\big)\Phi_{j} \right)
\end{equation}
\end{lem}
\begin{proof}
From Lemma \ref{mon}, we know that for any $t\ge T_{1}$,
\begin{equation}
M_{j}(t)-\sum_{i=1}^{j}M[P_{i}]\le Ce^{-2\varpi t}.
\end{equation}
By summing this fact with the fact from the Lemma \ref{1F2}, we obtain
that for any $t$ large enough: 
\begin{equation}
\mathcal{H}_{j}(t)-\sum_{i=1}^{j}F[P_{i}]-2\big(b_{j}^{2}-a_{j}^{2}\big)\sum_{i=1}^{j}E[P_{i}]
-\big(a_{j}^{2}+b_{j}^{2}\big)^{2}\sum_{i=1}^{j}M[P_{i}] \le  Ce^{-2\varpi t}.\label{1eq:-164} 
\end{equation}

From \eqref{1eq:-164} and Lemma \ref{1lem:dl}, we obtain
the desired conclusion for any $t\ge T_{4}$.
\end{proof}

Let us recall that $H_j$ enjoy a crucial coercivity property:

\begin{prop}[Coercivity of $H_{j}$]
\label{1lem:coerc-1}There exists $\mu>0$, and $T_5 \ge T_4$
such that, for $t \ge T_5$,
\begin{equation}
H_{j}(t)\ge \mu \int ( w_{xx}^2 + w_x^2 + w^2) \Phi_j(t) -\frac{1}{\mu}\bigg(\int\widetilde P_{j} w \sqrt{\Phi_{j} }\bigg)^{2}.\label{1coerc}
\end{equation}
\end{prop}

\begin{proof}
One can argue as in the proof of Proposition 4.10, Step 7 in \cite{key-49} (see also Proposition 2.13). (We choose to keep the $L^2$ scalar product of $\widetilde P_{j}$ with $w \sqrt{\Phi_{j}}$ because coercivity is derived from the original coercivity (related to the linearization around soliton or breather of the relevant conservation law see \cite{key-49,key-1}) via localization argument on $w \sqrt{\Phi_{j}}$. Obviously, we could have stated coercivity up to the scalar product of $\int \widetilde P_{j} w$).
\end{proof}

Using the above two results, and the induction hypothesis
we prove as in \cite[Step 6]{key-49} that for
$t \ge T_5$,
\begin{equation}
\label{coerc}
\int\big(w^{2}+w_{x}^{2}+w_{xx}^{2}\big)\Phi_{j}\le Ce^{-2\varpi t}+C\bigg(\int\widetilde P_{j}w  \sqrt{\Phi_{j} }\bigg)^{2}.
\end{equation}

\begin{rem}
Here the choice of modulating only by translation gives a writing simplification with respect to \cite{key-49}, where the scalar product only occured when $P_{j}$ was a breather: we do not need to make this distinction now.
\end{rem}

\subsubsection*{Step 5: Bound from above for $\displaystyle \left|\int\widetilde P_{j} w\sqrt{\Phi_{j}}\right|$}

We now show that the seemingly problematic scalar product $\displaystyle \int\widetilde P_{j}w$ is actually quadratic (which the second part of the induction hypothesis $\mathcal P_j$).

\begin{lem}
For any $t \ge T_5$,
\begin{equation}
\label{ps}
\bigg| \int\widetilde P_{j} w\sqrt{\Phi_{j}}\bigg| \le C e^{-2\varpi t}+ C \int\big(w^{2}+w_{x}^{2}\big)\Phi_{j}
\end{equation}
\end{lem}

\begin{proof}
The proof follows the lines of that of Step 7 of Proposition 4.10 in \cite[Section 4.2]{key-49}, and we only sketch it.

First we observe that it is enough to prove the same bound on 
\[ \left| \sum_{i=1}^{j}\int\widetilde P_i w \right| \]
because the induction hypothesis $\mathcal P_{j-1}$ \eqref{def:rec} takes care of the terms $i \le j-1$, and the localization $\sqrt{\Phi_j}$ causes an error of size $O(e^{-2\varpi t})$. Then, the idea is to go back to Lemma \ref{mon}, and work on $M_j$ and on $E_j, F_j$ separately: both make the scalar products $\int\widetilde P_i w$ appear, but with opposite signs.

On the one hand, due to \eqref{1eq:-155} and \eqref{est:mono_M}, we may deduce that
from Lemma \ref{mon} and Lemma \ref{1lem:dl} that
\begin{equation} \label{est:scal_1}
\sum_{i=1}^{j}\int\widetilde P_{i} w\le Ce^{-2\varpi t}.
\end{equation}

On the other hand, due to \eqref{1eq:-156}-\eqref{1eq:-157}, and using  the elliptic equation satisfied by $\widetilde P_{j}$ and the induction assumption, we may infer that for $t\ge T_{5}$,
\begin{align}
\MoveEqLeft-(1-\nu)(a_{j}^{2}+b_{j}^{2})^{2}\sum_{i=1}^{j}\int\widetilde P_{i }w\\ 
& =O(e^{-2\varpi t})+O\left(\int(w^{2}+w_{x}^{2})\Phi_{j}\right)\\
 & +\mathcal{F}_{j}(t)-\sum_{i=1}^{j}\left(F[P_{i}]+2(b_{j}^{2}-a_{j}^{2})E[P_{i}]+\nu(a_{j}^{2}+b_{j}^{2})^{2}M[P_{i}]\right).
\end{align}
Observe the factor $1-\nu >0$ that we placed: this make the functional $\mathcal F_j$ appear, and not merely the Lyapunov $\mathcal H_j$ (for which the linear terms cancel). Therefore, we are now in position of using the monotonicity \eqref{est:1F2}, and we obtain that
\begin{equation} \label{est:scal_2}
-(1-\nu)(a_{j}^{2}+b_{j}^{2})^{2}\sum_{i=1}^{j}\int\widetilde P_{i}w\le Ce^{-2\varpi t}+C\int(w^{2}+w_{x}^{2})\Phi_{j}.
\end{equation}
Estimates \eqref{est:scal_1} and \eqref{est:scal_2} yield together that
\begin{equation}
\left| \sum_{i=1}^{j}\int\widetilde P_{i} w \right|\le Ce^{-2\varpi t}+C\int(w^{2}+w_{x}^{2})\Phi_{j},
\end{equation}
which gives the desired conclusion.
\end{proof}

\subsubsection*{Step 6: Conclusion}

From \eqref{coerc} and \eqref{ps}, we deduce that for all $t \ge T_5$,
\begin{equation}
\int\big(w^2+w_x^2+w_{xx}^2\big)\Phi_j\le Ce^{-2\varpi t}.
\end{equation}
 Plugging this in \eqref{ps}, an immediate localization argument shows that
\[ \left| \int \widetilde P_{j} w \right| \le Ce^{-2\varpi t}. \]
Using also $\mathcal{P}_{j-1}$, this proves $\mathcal{P}_j$ \eqref{def:rec}, and the induction is complete.

\bigskip

Thus, for $j=J$, we get
\begin{equation}
\forall t \ge T_5, \quad \lVert w(t)\rVert_{H^2}\le Ce^{-\varpi t}.
\end{equation}
Now, recall that from $\| u(t) - P(t) \|_{H^2} \to 0$ and \eqref{mod_bou}, there hold
\begin{equation}
\lvert y_{1,k}(t)\rvert+\lvert y_{2,k}(t)\rvert+\lvert y_{0,l}(t)\rvert \to 0 \quad \text{as } t \to +\infty.
\end{equation}
Then, by integration from \eqref{mod_der} on $[t,+\infty)$ that for any $t \ge T_5$, 
\begin{equation}
\lvert y_{1,k}(t)\rvert+\lvert y_{2,k}(t)\rvert+\lvert y_{0,l}(t)\rvert \le Ce^{-\varpi t}.
\end{equation}

So, by the mean-value theorem,
\begin{equation}\begin{aligned}
\| u(t) - P(t) \rVert_{H^2}&\le \lVert w(t)\rVert_{H^2}+\lVert \widetilde P (t)-P(t)\rVert_{H^2}\\
&\le\lVert w(t)\rVert_{H^2}+\sum_{k=1}^K\big(\lvert y_{1,k}(t)\rvert+\lvert y_{2,k}(t)\rvert\big)
+\sum_{l=1}^L\lvert y_{0,l}(t)\rvert\\
&\le Ce^{-\varpi t}.
\end{aligned}\end{equation}
This concludes the proof of Proposition \ref{prop:uniq}.

\subsection{Proof of Theorem \ref{thm:uniq}}

Let us recall the uniqueness result of \cite{key-49} in the class of super-polynomial convergence to the profile.

\begin{prop}
\label{1lem:polyn}
Given breathers \eqref{br} and solitons \eqref{sol},
whose velocities satisfy \eqref{v_dist}, let $P$ be the sum of the considered solitons and breathers given
in \eqref{sum}.

There exists $N>0$ large enough such that the multi-breather from \cite{key-49}
$p \in \q C([T^*,+\infty),H^2)$ associated to $P$ is the \emph{unique} solution $u \in \mathcal{C}([T_0,+\infty),H^2(\mathbb{R}))$ of \eqref{mKdV} such that
\begin{equation}\begin{aligned}
\lVert u(t)-P(t) \rVert_{H^2} = O \bigg( \frac{1}{t^N} \bigg), \qquad \text{as } \ t \rightarrow +\infty.
\end{aligned}\end{equation}
\end{prop}

Theorem \ref{thm:uniq} is then an immediate consequence of Propositions \ref{1lem:polyn} and \ref{prop:uniq}.

\bibliography{bib}

\begin{thebibliography}{10}

\bibitem{key-1}
{\sc M.~A. Alejo and C.~Mu\~{n}oz}, {\em Nonlinear stability of {MK}d{V}
  breathers}, Comm. Math. Phys., 324 (2013), pp.~233--262.

\bibitem{key-61}
{\sc R.~C\^{o}te and X.~Friederich}, {\em On smoothness and uniqueness of
  multi-solitons of the non-linear {S}chr\"{o}dinger equations}, Comm. Partial
  Differential Equations, 46 (2021), pp.~2325--2385.

\bibitem{key-11}
{\sc R.~C\^{o}te and Y.~Martel}, {\em Multi-travelling waves for the nonlinear
  {K}lein-{G}ordon equation}, Trans. Amer. Math. Soc., 370 (2018),
  pp.~7461--7487.

\bibitem{key-16}
{\sc R.~C\^{o}te and C.~Mu\~{n}oz}, {\em Multi-solitons for nonlinear
  {K}lein-{G}ordon equations}, Forum Math. Sigma, 2 (2014), pp.~Paper No. e15,
  38.

\bibitem{EDM04}
{\sc K.~El~Dika and Y.~Martel}, {\em Stability of {$N$} solitary waves for the
  generalized {BBM} equations}, Dyn. Partial Differ. Equ., 1 (2004),
  pp.~401--437.

\bibitem{key-19}
{\sc C.~E. Kenig, G.~Ponce, and L.~Vega}, {\em Well-posedness and scattering
  results for the generalized {K}orteweg-de {V}ries equation via the
  contraction principle}, Comm. Pure Appl. Math., 46 (1993), pp.~527--620.

\bibitem{key-2}
{\sc Y.~Martel}, {\em Asymptotic {$N$}-soliton-like solutions of the
  subcritical and critical generalized {K}orteweg-de {V}ries equations}, Amer.
  J. Math., 127 (2005), pp.~1103--1140.

\bibitem{key-3}
{\sc Y.~Martel and F.~Merle}, {\em Multi solitary waves for nonlinear
  {S}chr\"{o}dinger equations}, Ann. Inst. H. Poincar\'{e} C Anal. Non
  Lin\'{e}aire, 23 (2006), pp.~849--864.

\bibitem{key-66}
{\sc M.~Ming, F.~Rousset, and N.~Tzvetkov}, {\em Multi-solitons and related
  solutions for the water-waves system}, SIAM J. Math. Anal., 47 (2015),
  pp.~897--954.

\bibitem{key-43}
{\sc T.~Perelman, A.~Fridman, and M.~El'yashevich}, {\em A modified korteweg-de
  vries equation in electrodynamics}, Sov. Phys. JETP, 39 (1974), pp.~643--646.

\bibitem{key-49}
{\sc A.~Semenov}, {\em On the uniqueness of multi-breathers of the modified
  korteweg-de vries equation}, accepted for publication in the Revista
  Matem{\'a}tica Iberoamericana,  (2021).

\bibitem{key-60}
{\sc F.~Valet}, {\em Asymptotic {$K$}-soliton-like solutions of the
  {Z}akharov-{K}uznetsov type equations}, Trans. Amer. Math. Soc., 374 (2021),
  pp.~3177--3213.

\end{thebibliography}
\bibliographystyle{siam}

\end{document}